%% file: sternbrocot.tex
\documentclass[12pt]{amsart}

\usepackage{tikz-cd}

\usepackage{lineno}

\include{mypackages}

\include{mymacros}

\let\emptyset\varnothing

\DeclareMathOperator{\gVect}{gVect}
\DeclareMathOperator{\Config}{Config}
\DeclareMathOperator{\PSL}{PSL}

\DeclareMathOperator{\B}{B}

\DeclareMathOperator{\Id}{Id}
\DeclareMathOperator{\op}{op}

\makeatletter
\def\imod#1{\allowbreak\mkern2.5mu({\operator@font mod}\,#1)}
\makeatother

\newcommand{\hohf}{\Ho_{\frac{1}{2}}(\eA)}
\newcommand{\hfcx}{\Kom_{\frac{1}{2}}^*(n)}
\newcommand{\hfcxA}{\Kom_{\frac{1}{2}}^*(\eA)}

\newcommand{\sls}{/}

\renewcommand{\bt}[0]{\B_3}
\newcommand{\pslz}[0]{\PSL(2,\ZZ)}
\newcommand{\slz}[0]{\SL(2,\ZZ)}
\newcommand{\sh}[0]{\mathfrak{s}}

\newcommand{\ptwT}[0]{\frac{1}{2} k^2}
\newcommand{\ptwQ}[0]{k}

\newcommand{\ptwTsq}[0]{k^2}
\newcommand{\ptwQsq}[0]{2k}

\newcommand{\ntwT}[0]{-\frac{1}{2} k^2}
\newcommand{\ntwQ}[0]{-k}

\newcommand{\pSHIFT}[0]{t^{\ptwT} q^{\ptwQ}}
\newcommand{\nSHIFT}[0]{t^{\ntwT} q^{\ntwQ}}

\newcommand{\kK}{\eK}

\newcommand{\kC}{\eC}
\newcommand{\kP}{\eE}
\newcommand{\kR}{\eR}
\newcommand{\kZ}{\eZ}
\newcommand{\kX}{\eX}

\theoremstyle{plain}
\newtheorem{property}{Property}

\begin{document}

\title[Categorical Representations of The Modular Group]{Categorical Representations of The Modular Group}
\author[Benjamin Cooper]{Benjamin Cooper}
\address{Universit\"{a}t Z\"{u}rich, Winterthurerstrasse 190, CH-8057 Z\"{u}rich}
\email{benjamin.cooper\char 64 math.uzh.ch}

\begin{abstract}
  Actions of the modular group on categories are constructed.  A
  hyperelliptic involution is used to convert the braid representations
  underlying Khovanov homology to representations of the modular group.
\end{abstract}

\maketitle

\section{Introduction}\label{introsec}
The study of group actions on categories is a rich topic which arises in
many areas of mathematics. Such group actions appear when topological field
theories are present. An $n$-dimensional topological field theory $\rZ$
associates to each $(n-2)$-manifold $\Sigma$ a category $\rZ(\Sigma)$ upon
which the mapping class group $\Gamma(\Sigma)$. In
particular, there are equivalences:
$$ F_g : \rZ(\Sigma) \to \rZ(\Sigma) \conj{ for each } g \in \Gamma(\Sigma).$$
When $n=4$, one expects many examples of mapping class groups
$\Gamma(\Sigma)$ acting on categories $\rZ(\Sigma)$, see \cite{KH,
  Roberts}. While there are examples of categorical braid group actions in
the literature \cite{KoSei, Rou, ThoSei, Webster}, there are very few
examples for surfaces of genus greater than zero; although, see \cite{LOT}.

In this paper we present a means by which the categories underlying the
Khovanov homology of knots and links can be used to produce categories which
are representations of the modular group $\slz$ (the mapping class group of
the torus). A hyperelliptic involution is used to reduce the mapping class
group of the torus to the mapping class group of the four punctured
sphere. This identification allows us to construct several families of
modular group representations from braid group representations.

The theory of modular functors shows that it is possible to construct
representations for surfaces of arbitrary genus when given a family of braid
group representations and a representation of the modular group satisfying
some compatibility conditions \cite{BaK, MS, Walker}.  In this sense, the
special case of $\slz$ actions is an important component of the most general
case.

\subsection*{Organization}
In Section \ref{algsec}, definitions related to knot homology and group
actions are recalled. In Section \ref{braidsec}, we review information about
braid groups and mapping class groups. In Section \ref{gensec}, we construct
families of categories $\kR_{n,k}$ and $\kK_n$ which support $\slz$ actions.

\section{Algebraic Background}\label{algsec}

In this section we collect some algebraic background information.  We begin
by recalling the differential graded categories $\Kom^*_{\frac{1}{2}}(n)$
and the existence of certain chain complexes $P_{n,k}$ within
$\Kom^*_{\frac{1}{2}}(n)$. The definition of a group action on a category is
given and a construction for reducing the gradings is introduced.

\subsection{Cobordisms}\label{cobsec}
Dror Bar-Natan's graphical formulation \cite{DBN} of the Khovanov
categorification \cite{KH} consists of a series of categories:
$\Cob(n)$ and $\Kom(n)$. The construction in Section \ref{gensec} will take
place within categories related to $\Kom(n)$. Definitions are reviewed
below. They are the same as those appearing in \cite{CH, CK} with the
exception that we adopt the Kauffman bracket grading convention from
Rozansky \cite{R3}.

\subsubsection{Categories of cobordisms}
Let $R$ be a field.

\begin{definition}\label{cobdef}
  The category of {\em $\mathfrak{sl}_2$-cobordisms} with $2n$ boundary points
  will be denoted by $\Cob(n)$.
\end{definition}

In more detail, there is a pre-additive category $\PCob(n)$ whose objects
are formally $\ZZ$-graded diagrams. Each diagram is given by an isotopy
class of 1-manifolds embedded in a disk $D^2$ relative to $2n$ boundary
points.  There is a functor
$$q : \PCob(n) \to \PCob(n)$$ 
which increases the grading by one. The morphisms of the
category $\PCob(n)$ are given by $R$-linear combinations of isotopy classes
of orientable cobordisms embedded in $D^2 \times I$, see \cite{DBN}.

The \emph{internal degree} of a cobordism $C : q^i A \to q^j B$ is the sum
of its topological degree and its $q$-degree:
$$\deg(C) = \deg_{\chi}(C) + \deg_q(C).$$
The {\em topological degree} $\deg_{\chi}(C) = \chi(C) - n$ is given by the
Euler characteristic of $C$ and the {\em $q$-degree} $\deg_q(C) = j - i$ is given
by the relative difference in $q$-gradings.

We impose the relations implied by the requirement that the circle is isomorphic to the direct sum of two empty sets,
\begin{equation}\label{circeq}
\CPic{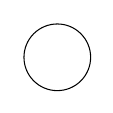}\quad\cong\quad q\emptyset \oplus q^{-1} \emptyset,
\end{equation}
using maps of internal degree zero (see \cite{DBN2}). In doing so, we obtain
the category $\Cob(n)$ as a quotient of the additive closure
$\Mat(\PCob(n))$ of $\PCob(n)$. 

There is a composition: 
$$\otimes : \Cob(n) \times \Cob(n) \to \Cob(n).$$
It is defined by gluing diagrams and cobordisms. The unit $1_n$
with respect to this composition is a diagram consisting of $n$ parallel
lines.

\subsubsection{Categories of half-graded chain complexes}\label{chcxsec}

\begin{definition}\label{komdef}
  If $\eA$ is an additive category then $\hfcxA$ will denote the {\em
    differential graded category of $\frac{1}{2}$-graded chain complexes} in
  $\eA$.
\end{definition}

In more detail, the objects of $\hfcxA$ are chain complexes $C =
(C_i,d^C_i)_{i\in\frac{1}{2}\ZZ}$ where $C_i\in \eA$ and $d^C_i : C_i \to
C_{i+1}$. Each chain complex $C$ is bounded from below: $C_i = 0$ for $i$ sufficiently small.

Suppose that $C = (C_i,d^C_i)_{i\in\frac{1}{2}\ZZ}$ and $D =
(D_i,d^D_i)_{i\in\frac{1}{2}\ZZ}$ are two chain complexes in $\hfcxA$. Then a
{\em map} $g : C \to D$ of {\em $t$-degree} $l$ is a collection $g = \{ g_i
: C_i \to D_{i+l} \}_{i\in \frac{1}{2}\ZZ}$ of maps in $\eA$. The vector
space of maps with $t$-degree $l$ will be denoted by $\Hom^l(C,D)$.  There
is a composition
$$\circ : \Hom^l(C, D) \ott \Hom^m(D, E) \to \Hom^{l+m}(C, E)$$
and there are identity maps $1_{C} = \{ 1_{C_i} : C_i \to C_i \}$ of degree
zero.

There is a differential $d : \Hom^l(C, D) \to \Hom^{l+1}(C, D)$ which is
defined on a map $g = \{g_i\}$ of $t$-degree $l$ by the formula:
\begin{equation}\label{homdiffeq}
(dg)_{i+1} = g_{i+1} \circ d^C_i + (-1)^{l+1} d^D_{i+l} \circ g_i.
\end{equation}
If $g \in \Hom^l(C,D)$ and $h \in \Hom^m(C,D)$ are maps of $t$-degree $l$ and $m$ respectively, then the differential $d$ satisfies the Leibniz rule:
$$d(h \circ g) = h \circ (dg) + (-1)^{l} (dh) \circ g.$$
If $1_C : C \to C$ is the identity map then $d(1_C) = 0$. Let $\Hom^*(C,D) =
\prod_l \Hom^l(C,D)$ be the maps from $C$ to $D$. Then the pairs
$(\Hom^*(C,D),d)$ are chain complexes and the category $\hfcxA$ is a
differential graded category \cite{Toen}.

\begin{definition}
The {\em homotopy category} $\hohf$  of $\hfcxA$ has the same objects as the category $\hfcxA$. The maps in $\hohf$ are given by the zeroth homology groups of the maps in $\hfcxA$:
$$\Hom_{\hohf}(C,D) = \Hm^0(\Hom^*(C,D)),$$
see \cite[Def. 1]{Toen}. 
\end{definition}

Two objects $C$ and $D$ in $\hfcxA$ are {\em homotopic}, $C\simeq D$, when they are isomorphic in the homotopy category $\hohf$. A chain complex $C$ is {\em contractible} when $C\simeq 0$.


\begin{definition}
The {\em category of $\frac{1}{2}$-graded chain complexes of cobordisms} $\hfcx$ is the differential graded category of $\frac{1}{2}$-graded chain complexes in $\Cob(n)$.
$$\hfcx = \Kom^*_{\frac{1}{2}}(\Cob(n))$$
\end{definition}

\subsubsection{Tangle invariants}\label{tangsec}

\begin{definition}\label{skeindef}
The {\em skein relation} associated to a crossing is the cone complex:
\begin{equation}\label{skeineq}
\CPic{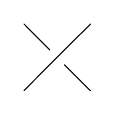} = t^{\frac{1}{2}} \CPic{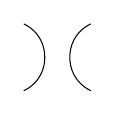} \to t^{-\frac{1}{2}} \CPic{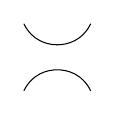}.
\end{equation}
\end{definition}

The definition above allows one to associate chain complexes in the category
$\hfcx$ to diagrams of framed $(n,n)$-tangles. Any two such diagrams which
differ by Reidemeister moves 2 or 3 are assigned to homotopy equivalent
chain complexes. The first Reidemeister move results in a degree shift, see
Property \ref{abstwprop} in Section \ref{twandtangsec}.

\subsection{Projectors}\label{projectorsec}
In \cite{CK}, V. Krushkal and the author defined a special chain complex
$P_n$. This construction is summarized by the theorem below. See also \cite{FSS, R1}.

\begin{theorem}\label{uprojectorthm}
There exists a chain complex $P_n \in \Kom^*_{\frac{1}{2}}(n)$ called the \emph{universal
  projector} which satisfies

\begin{enumerate}
\item $P_n$ is positively graded with differential having internal degree zero.
\item The complex $P_n\ott D$ is contractible for any diagram $D$ which is not the identity $1_n$.
\item The identity object appears only in homological degree zero and only once.
\end{enumerate}
\noindent Properties (1)--(3) characterize the projector $P_n$ uniquely up to homotopy.
\end{theorem}

In \cite{CH}, M. Hogancamp and the author constructed a series of projectors
$P_{n,k}\in\Kom^*_{\frac{1}{2}}(n)$ where $k = n, n-2, n-4, \ldots, n \imod{2}$. (The bottom
projector $P_{n,n\imod{2}}$ is $P_{n,0}$ when $n$ is even and $P_{n,1}$ when
$n$ is odd.) These {\em higher order projectors} are characterized by
properties analogous to those mentioned in the theorem above.

\begin{definition}
  Suppose that $C = (C_i,d^C_i)_{i\in\frac{1}{2}\ZZ}$ is a chain complex in
  $\Kom^*_{\frac{1}{2}}(n)$ and let $C_i = \oplus_j C_{i,j}$ so that each
  object $C_{i,j}\in\PCob(n)$ is a diagram. Then the \emph{through-degree}
  $\tau(C)$ is the maximum over all of the diagrams $C_{i,j}$ of the minimum
  number of vertical strands in any cross section of $C_{i,j}$.
\end{definition}

In other words, $\tau(C)$ is the largest number of lines connecting the bottom to
the top of any diagram $C_{i,j}$ in $C$. Since vertical lines can only be added
or removed in pairs, we have: 
$$\tau(C)\in\{ n,n-2,n-4,\ldots,n \imod{2}\}.$$

\begin{theorem}\label{hothm}
  For each positive integer $n$, there is a series of {\em higher order
    projectors} $P_{n,k}\in\Kom^*_{\frac{1}{2}}(n)$ where $k = n, n-2, n-4, \ldots, n
  \imod{2}$ which satisfy the properties:

\begin{enumerate}
\item The $P_{n,k}$ have through-degree $\tau(P_{n,k}) =k$.
\item If $D$ is diagram with through-degree $\tau(D) < k$ then
$$D\otimes P_{n,k} \simeq 0\conj{ and } P_{n,k} \otimes D \simeq 0.$$
\end{enumerate}

\noindent Together with a normalization axiom, analogous to (3) in Theorem
\ref{uprojectorthm} above, these properties characterize each object
$P_{n,k}$ uniquely up to homotopy.
\end{theorem}

The higher order projectors satisfy orthogonality and idempotence relations. They also fit together to form a convolution chain complex that is homotopy equivalent to the identity object, \cite[\S 8]{CH}.
\begin{equation}\label{decompeq}
P_{n,l}\otimes P_{n,k}\simeq \d_{lk}P_{n,k} \conj { and } 1_n \simeq P_{n,n \imod{2}} \rightarrow \cdots \rightarrow P_{n,n-4} \rightarrow P_{n,n-2}\rightarrow P_{n,n}
\end{equation}

\vskip .1in

\begin{remark}
  The objects $P_n$ and $P_{n,k}$ were originally defined in categories
  $\Kom(n)$. By construction, every object of $\Kom(n)$ is an object of
  $\Kom^*_{\frac{1}{2}}(n)$. This identification allows us to state the
  theorems as we have above.
\end{remark}

\subsubsection{Tangles and twists}\label{twandtangsec}
If the projector $P_{n,k}$ is represented by a box then it satisfies the two
properties below.

\begin{property}{(drags through tangles)}\label{dragprop}
$$\CPic{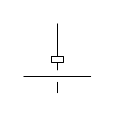} \simeq \CPic{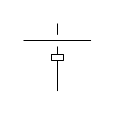}\conj{ and } \CPic{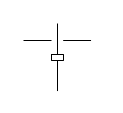} \simeq \CPic{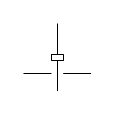}$$
\end{property}

\begin{property}{(absorbs twists)}\label{abstwprop}
$$\CPic{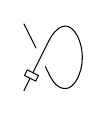} \simeq\quad \pSHIFT \CPic{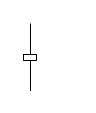}\hspace{-.185in}\conj{ and }\CPic{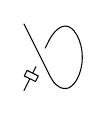} \simeq\quad \nSHIFT  \CPic{line-proj}$$
\end{property}

For the degree shifts above to be accurate, the skein relation must be given
by Equation \eqref{skeineq}. When $n=1$ and $k=1$, Property 2 determines the
grading shift associated to the first Reidemeister move.

\subsection{Group actions on categories}\label{groupactsec}
\begin{definition}
If $G$ is a group and $\kC$ is a differential graded category then an {\em action
of $G$ on $\kC$} is a homomorphism from the group $G$ to the category $\Endo(\eC)$ of 
functors from $\kC$ to $\kC$:
$$F : G \to \Endo(\kC)\conj{ such that } F(gh)\simeq F(g)\circ F(h). $$
For discussion of dg functors see \cite[\S 2.3]{Toen}.  
\end{definition}

If $G = \inp{ S : R}$ is a presentation for the group $G$ then a group
action is obtained by specifying functors $F_s, F_{s^{-1}} : \kC \to \kC$
for each generator $s\in S$, and homotopy equivalences
\begin{equation}\label{eqeq}
F_{s_1} \circ F_{s_2} \circ \cdots \circ F_{s_N} \simeq F_{s'_1} \circ
F_{s'_2} \circ \cdots \circ F_{s'_M} 
\end{equation}
whenever the words $s_1 s_2 \cdots s_N$ and $s'_1 s'_2 \cdots s'_M$ coincide
by virtue of the relations in $R$. To accomplish this, it suffices to find
homotopy equivalences, $F_{s_1} \circ \cdots \circ F_{s_N} \simeq \Id_\eC$
for each $r\in R$, $r = s_1\cdots s_N$.

When a group $G$ acts on a dg category $\eC$, there is an induced action of
$G$ on the homotopy category $\Ho(\eC)$ of $\eC$. If $\Ho(\eC)$ is a
triangulated category then the action of $G$ on $\Ho(\eC)$ agrees with
definitions found in the references \cite{KoSei, Rou, ThoSei}.

The notion of group action that is given above is sometimes called
\emph{weak}. A strong action is one in which the homotopy equivalences
\eqref{eqeq} are uniquely determined. We will only prove that the actions
defined here are weak, see Remark \ref{remark3}.

In Section \ref{gensec}, we will construct $\SL(2,\ZZ)$ actions on
categories $\kR_{n,k}$ and $\kK_n$.

\subsection{Reducing the grading}\label{cyclicgrsec}

\newcommand{\qu}{\sls \shft}
\newcommand{\Cqu}{\eC\qu}
\newcommand{\shft}{t^{\frac{n}{2}}q^m}

In this section we introduce cyclically graded dg categories $\eC\qu$. If
$\eC$ is a $\frac{1}{2}$-graded dg category, such as
$\Kom^*_{\frac{1}{2}}(\eA)$, then there are functors
\begin{equation}\label{shifteq}
t^{\frac{1}{2}},q : \eC \to \eC,
\end{equation}
which increase the $t$ and $q$ gradings by $\frac{1}{2}$ and
$1$ respectively. The categories $\eC\qu$ are obtained by collapsing the
grading so that the identity:
$$\shft \cong \Id_{\eC\qu}$$
holds in the category $\Endo(\eC\qu)$. This
is accomplished by extending a technique used to define $\ZZ/2$-graded dg
categories \cite[\S 5.1]{Dyckerhoff}.

\begin{defn}\label{anmalg}
  For each pair of integers, $n,m\in \ZZ$, there is a $q$-graded differential
  graded algebra $(L_{n,m},d)$ which is given by
$$L_{n,m} = R[t^{\frac{1}{2}}, q]/(t^{\frac{n}{2}}q^m=1) \conj{ and } d = 0.$$
The ring $R$ is chosen to be the ground field of the category $\eC$.  The grading of
$L_{n,m}$ is determined by the table below.
\begin{center}\begin{tabular}{c||cc} & $\deg_q$ & $\deg_t$ \\
\hline
$q$ & $1$ & $0$\\
$t^{\frac{1}{2}}$ &$0$  & $\frac{1}{2}$\\
\end{tabular}\end{center}
\end{defn}

\begin{defn}
Let $\eC$ be a $\frac{1}{2}$-graded dg category. An {\em $L_{n,m}$-module in $\eC$} is a pair $(C,f)$ consisting of an object $C \in \eC$ and a map of graded dg algebras:
$$f : L_{n,m} \to \Endo(C).$$
The {\em category $\eC\qu$ of $L_{n,m}$-modules in $\eC$} is the dg category consisting of $L_{n,m}$-modules and $L_{n,m}$-equivariant maps.
\end{defn}

In more detail, the objects of $\eC\qu$ are pairs $(C, f)$ consisting
of objects $C\in \eC$ and maps $f : L_{n,m} \to \Endo(C)$ of
differential graded algebras which preserve the gradings. A map $h : (C,f) \to (D,g)$ of degree $l$
is a map $h : C \to D$ of degree $l$ in $\eC$ which commutes with the
action of $L_{n,m}$ on $C$ and $D$ respectively:
\begin{equation}\label{commeq}\begin{tikzcd}
C\arrow{r}{h}\arrow{d}{f(a)}&D\arrow{d}{g(a)}\\
C\arrow{r}{h}&D
\end{tikzcd}\end{equation}
for all $a \in L_{n,m}$.

The functors $t^{\frac{1}{2}} , q : \eC \to \eC$ from Equation
\eqref{shifteq} induce functors $t^{\frac{1}{2}}, q : \eC\qu \to
\eC\qu$. For example, if $(C,f)$ is an object of $\eC\qu$ then
$t^{\frac{k}{2}}q^l(C,f) = (t^{\frac{k}{2}} q^l C, \vp \circ f)$ where $\vp$
is the natural isomorphism $\Endo(C) \to \Endo(t^{\frac{k}{2}}q^lC)$.

The proposition below shows that the equation $t^{\frac{n}{2}} q^m \cong
\Id_{\eC\qu}$ holds in $\Endo(\eC\qu)$.

\begin{proposition}
Suppose that $t^{\frac{1}{2}}, q : \Cqu \to \Cqu$ denote the functors which increase the $t$-degree and $q$-degree in $\Cqu$ by $\frac{1}{2}$ and $1$ respectively. Then there is an isomorphism,
$$\eta : t^{\frac{n}{2}} q^m \cong \Id_{\eC\qu}$$
in the functor category $\Endo(\Cqu)$.
\end{proposition}
\begin{proof}
  The definition of $\Cqu$ implies that for each object $(C,f)\in \Cqu$ there are maps
  $f(t^{\frac{1}{2}}) : C\to C$ and $f(q) : C\to C$. These maps determine
  isomorphisms: $\eta_C : t^{\frac{n}{2}} q^m (C,f) \to (C,f)$.  If $h : (C,f) \to
  (D,g)$ is a map in $\Cqu$ then Equation \eqref{commeq} implies that
  $h$ commutes with $\eta_C$. So the collection $\eta = \{\eta_C\}$ is a
  natural transformation $t^{\frac{n}{2}} q^m \to \Id$.
The map $\eta$ is the isomorphism in the statement above.
\end{proof}

Suppose that $\gVect_R$ is the category of graded vector spaces over
$R$. Then the example below illustrates the material introduced above.

\begin{example}
  An $L_{0,m}$-module in the dg category $\Kom^*_{\frac{1}{2}}(\gVect_R)$ is
  a chain complex $C$ on which the grading shift functor $q^m : C\to C$ acts
  by identity. Similarly, an $L_{n,0}$-module $C$ in the category
  $\Kom^*_{\frac{1}{2}}(\gVect_R)$ is a chain complex on which the functor
  $t^{\frac{n}{2}} : C\to C$ acts by identity.

  The categories $\Kom^*_{\frac{1}{2}}(\gVect_R)\sls t^0q^m$ and
  $\Kom^*_{\frac{1}{2}}(\gVect_R)\sls t^{\frac{n}{2}}q^0$ consist of
  $q^m$-cyclic and $t^{\frac{n}{2}}$-cyclic chain complexes respectively.
\end{example}

\vskip .1in

There is a {\em forgetful functor} $U : \eC\qu \to \eC$ which is determined
by the assignment $(C,f) \mapsto C$.  The functor $U$ has a left adjoint $P
: \eC \to \eC\qu$ which sends a chain complex $C$ to $P(C) = L_{n,m}\ott_R
C$. If $\pi : \frac{1}{2}\ZZ\times\ZZ \to \frac{1}{2}\ZZ\times\ZZ/(n,m)$ is
the quotient map then
$$P(C)_{i,j} = \bigoplus_{ (k,l)\in \pi^{-1}(i,j) } C_{k,l}$$
where $C_{k,l}$ denotes the value of $C$ in $t$-degree $k$ and $q$-degree $l$. 
The proposition below records this information.

\newcommand{\komatn}{\Kom^*_{\frac{1}{2}}(\eA)\qu}
\newcommand{\hotn}{\Ho_{\frac{1}{2}}(\eA)\qu}

\begin{proposition}\label{adjprop}
There is an adjunction:
$$P : \eC \leftrightarrows \Cqu : U.$$
\end{proposition}
The proof that the free $L_{n,m}$-module functor $P = L_{n,m}\ott_{R} -$ is left adjoint to the forgetful functor $U$ is standard, see \cite[\S IV.2]{MacLane}.

In the remainder of this section we specialize to chain complexes.

\begin{remark}
  If $U(C)$ and $U(D)$ are cyclically graded complexes then $U(C)\ott U(D)$
  is a cyclically graded complex. If $\eA$ is a monoidal category then
  $\komatn$ is a monoidal category.
\end{remark}

\begin{proposition}
The homotopy category $\hotn$ of $\komatn$ is a triangulated category.  
\end{proposition}
The proof assumes familiarity with the reference \cite{BK}.
\begin{proof}
  The functors $P$ and $U$ commute with totalization. 

  In more detail, let $\eC = \komatn$. Given a twisted complex $E =
  \{E_i, r_{ij}\}$ over $\eC$ \cite[\S 1 Def. 1]{BK}. Let $\a(E) : \eC^{\op}
  \to \Kom$ be the functor determined by $E$ \cite[\S 1 Def. 3]{BK}. By
  definition,
$$\a(E)(X) = \oplus_i t^i \Hom_{\eC}(X,E_i).$$
The differential of this chain complex is the sum $d + Q$ where $d$ is the
$\Hom$-differential from Equation \eqref{homdiffeq} and $Q=(r_{ij})$.  

Let $T(E) = \oplus_i t^i E_i \in \eD$ be the chain complex with differential
$(r_{ij})$. Then the chain complex $T(E)$ represents the functor
$\a(E)$. This is because the functor $h_{T(E)} : \eC^{\op}\to\Kom$ associated to $T(E)$ by
the Yoneda embedding is given on objects $X$ by
$$h_{T(E)}(X) = \Hom(X,T(E)) = \Hom(X,\oplus_i t^i E_i) \cong\oplus_i t^i \Hom(X, E_i).$$
Checking that the differentials agree implies that $h_{T(E)}(X) \cong
\a(E)(X)$. This isomorphism does not depend on $X$, so $h_{T(E)} \cong
\a(E)$.

Since each twisted complex $E$ is represented by a chain complex $T(E)$ in
$\eC$, the category $\eC$ is pre-triangulated \cite[\S 3 Def. 1]{BK}. It
follows that the homotopy category $\Ho(\eC) = \hotn$ is triangulated
\cite[\S 3 Prop. 2]{BK}.
\end{proof}

The argument above also applies to subcategories.

\begin{corollary}
  Cyclic reductions of pre-triangulated subcategories $\eS\subset
  \Kom^*_{\frac{1}{2}}(\eA)$ have triangulated homotopy categories
  $\Ho_{\frac{1}{2}}(\eS)\qu$.
\end{corollary}

\vskip .1in

The proof of the main theorem shows that the modular group relations hold in
a certain category $\kP_{n,k}$ up to a grading shift $\sh_k :\kP_{n,k}\to \kP_{n,k}$. The
construction introduced in this section allows us to reduce the grading so
that the relations of the modular group hold without the grading shift.

\section{Braid groups and the modular group}\label{braidsec}

In this section we recall the relationship between the braid group $\bt$ and
the modular group $\pslz$. Notation established here will be used later.

\subsection{Braid groups}\label{braidgpssec}
If $D_n$ is the closed disk with $n$ punctures then the {\em braid group} $\B_n$
on $n$ strands is given by the path components of the group of boundary and
orientation preserving diffeomorphisms of $D_n$:
$$\B_n = \pi_0(\Diff^+(D_n,\partial D_n)).$$
Alternatively, suppose that $\Config_n(D^2)$ is the configuration space of $n$ distinct points in the unit disk. Then the symmetric group $S_n$ acts on $\Config_n(D^2)$ by permuting these points and the braid group $\B_n$ is the fundamental group of the quotient,
$$\B_n = \pi_1(\Config_n(D^2)/S_n,x_0),$$
see \cite{Birman, FM}. Each element $[\gamma] \in\B_n$ determines a loop
$\ga\in[\ga]$, $\ga : (I,\partial I)\to (\Config_n(D^2)/S_n,x_0)$. Such a
map $\ga$ factors as $\ga(t) = (\ga_1(t),\ldots,\ga_n(t))$ where $\ga_i : I
\to D^2$. The composition of the union $\coprod_i \ga_i$ with the fold map
is an embedding $\tilde{\ga} : \coprod_i I \hookrightarrow D^2\times I$
which determines a braid $\im(\tilde{\ga})$ in 3-space.  Each such braid can
be represented by a diagram consisting of a sequence of crossings $\sigma_i$
for $1\leq i < n$.

When $n=3$, a presentation for the braid group in terms of these generators
is given by
$$\bt = \inp{\s_1, \s_2 : \s_1 \s_2 \s_1 = \s_2\s_1\s_2 }.$$

\subsection{Modular groups}\label{pslzsec}

The {\em modular group} $\slz$ is the group of $2\times 2$ matrices
with integer coefficients and determinant $1$:
$$\slz = \left\{ \left(\begin{array}{cc} a & b \\ c & d \end{array} \right) : a,b,c,d\in\ZZ\conjj{and} ad-bc = 1\right\}.$$
It is also the mapping class group of the torus $T$, see \cite{FM}. If $I\in
\slz$ is the identity matrix then the center $\kZ(\slz) \subset \slz$ is the
subgroup $\{ I, -I \}$. The {\em projective modular group} $\pslz$ is the quotient
$$\pslz = \slz/\{I, -I\}.$$
Choosing representatives for the elements $\{I,-I\}$ yields a $\ZZ/2$-action on the torus
with quotient: $S^2_4 = T/(\ZZ/2)$, the four punctured sphere. The group
$\pslz$ is the mapping class group of $S^2_4$. It can be shown that $\pslz$
is isomorphic to the free product $\ZZ/2 * \ZZ/3$; equivalently:
$$\pslz \cong \inp{ a, b : a^2 = 1, b^3 = 1}.$$

\subsection{Relating $\bt$ and $\pslz$}\label{relatingsec} Denote by $q_n$ the braid:
$$q_n = \s_{n-1} \s_{n-2} \cdots \s_1 \in \B_n.$$
For instance:
$$q_3 = \CPic{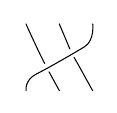}.$$ 
Multiplying $q_n$ by itself $n$ times produces the full twist, $T_n = q_n^n$. The equation,
\begin{equation}\label{commuteeq}
\s_i T_n = T_n \s_i \conj{for} 1\leq i < n,
\end{equation}
can be seen by dragging $\s_i$ through from the bottom of $T_n$ to the top of
$T_n$. It follows that $T_n$ is contained in the center $\kZ(\B_n) \subset
\B_n$ of the braid group. In fact, $\inp{T_n} = \kZ(\B_n)$, see \cite{chow}.

The following proposition tells us that the quotient $\B_3/\inp{T_3}$ is related to the modular group.

\begin{proposition}\label{btpslzthm}
The quotient of $\B_3$ by the subgroup generated by the full twist $T_3$ is isomorphic to the projective modular group $\pslz$.
$$\B_3/\inp{T_3} \cong \pslz$$

\end{proposition}

\begin{proof}
If we set $a = \s_1 \s_2 \s_1$ and $b = \s_1 \s_2$ then
$$\B_3/\inp{T_3} = \inp{\s_1, \s_2 : \s_1 \s_2\s_1 = \s_2\s_1\s_2, T_3 = 1} \cong \inp{a,b : a^2 = 1, b^3 = 1} \cong \pslz.$$
\end{proof}

The groups $\B_n/\inp{T_n}$ are always related to mapping class groups of
punctured spheres, see \cite[\S 4.2]{Birman}.

\section{The Construction}\label{gensec}

In this section we construct sequences of categories $\kR_{n,k}$ and
prove that each category supports an action of the modular group. We begin
by defining categories $\kP_{n,k}$ as extensions of $\Kom^*_{\frac{1}{2}}(3n)$ from
Section \ref{cobsec}. The categories $\kR_{n,k}$ are obtained from
$\kP_{n,k}$ using the construction in Section \ref{cyclicgrsec}. In Theorem
\ref{pslactsthm} below, it is shown that there is an action of the group
$\pslz$ on each category $\kR_{n,k}$. This induces an action of the modular
group on $\kR_{n,k}$.

The first step is to define a chain complex $E_{n,k}$ consisting of a
composition of projectors. This object is illustrated in the proof
of Theorem \ref{pslactsthm}.

\begin{definition}{($E_{n,k}$)}
  Suppose that $P_{n,k}$ is the higher order projector from Theorem \ref{hothm}.
Then set
$$E_{n,k} = (P_{n,k} \sqcup P_{n,k} \sqcup P_{n,k}) \otimes P_{3n,k}.$$
\end{definition}

\begin{definition}{($\kP_{n,k}$)}\label{apndef}
The category $\kP_{n,k}$ is the full dg subcategory of $\Kom^*_{\frac{1}{2}}(3n)$
consisting of objects which are homotopy equivalent to those of the form $C
\otimes E_{n,k}$. 
\end{definition}

Each object in the category $\kP_{n,k}$ is bigraded. One can shift either
the $t$-degree or the $q$-degree of each object using the functors
$t^{\frac{1}{2}}$ or $q$. Property \ref{abstwprop} in Section
\ref{projectorsec} suggests the following definition.

\begin{definition}{($\sh_k$)}
  The shift functor $\sh_k : \kP_{n,k} \to \kP_{n,k}$ is the
  composition $\sh_k = t^{\ptwTsq} q^{\ptwQsq}$.
\end{definition}

If $C\in\kP_{n,k}$ is a chain complex then
$$\deg_t(\sh_k(C)) = \deg_t(C) + \ptwTsq \conj{and} \deg_q(\sh_k(C)) = \deg_q(C) + \ptwQsq.$$

\begin{definition}{($\kR_{n,k}$)}\label{arndef}
The category $\kR_{n,k}$ is obtained using the construction from Section \ref{cyclicgrsec}:
$$\kR_{n,k} = \kP_{n,k}\sls \sh_k.$$
\end{definition}

\begin{theorem}\label{pslactsthm}
The group $\pslz$ acts on the category $\kR_{n,k}$.
\end{theorem}
\begin{proof}
Using the presentation  $\pslz = \inp{ a, b : a^2 = 1, b^3 = 1}$ from Section \ref{pslzsec}, we will construct a functor:
$$\pslz \to \Endo(\kR_{n,k}) \conj{where} g\mapsto F_g.$$
Associated to the generators $a$ and $b$ are the chain complexes that are
determined by the tangles:
$$a=\CPic{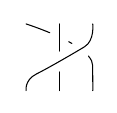} \conj{and} b = \CPic{cubebraid3},$$
where each line represents an $n$-cabled strand. There are functors, $F_a,
F_b :\kR_{n,k} \to \kR_{n,k}$, given by tensoring on the left. To the
identity element $1 \in\pslz$, we associate the functor $F_1$ where $1=1_{3n}$ is
the diagram consisting of $3n$ vertical strands. For inverses, $a^{-1}$ and
$b^{-1}$, we use the functors $F_{a}$ and $F_{b^2}$. To a word
$w\cdot w' \in \inp{a,b}$, we associate the functor $F_w\circ F_w'$.  In
this manner, we obtain a functor $F_g : \kR_{n,k} \to \kR_{n,k}$ for each
word $g\in\inp{a,b}$.

We proceed by checking that the relations $a^2 = 1$ and $b^3 = 1$ hold.
Since the braids $a^2$ and $b^3$ are both full twists with some framing
dependency, it suffices to check that this full twist is homotopy equivalent
to the identity object in the category $\kR_{n,k}$.

Consider what happens when the tangles $a^2$ or $b^3$ are glued onto
$E_{n,k}$. As framed tangles, the full twist $T_3$ is isotopic to the tangle
pictured on the right below.

$$\hspace{-1.5in}\CPic{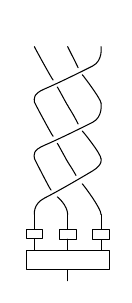}\quad\quad\simeq\CPic{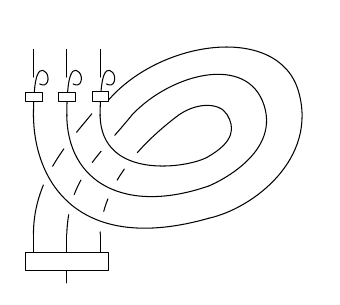} \hspace{1.2in}\simeq \quad \left(\pSHIFT\right)^3 \CPic{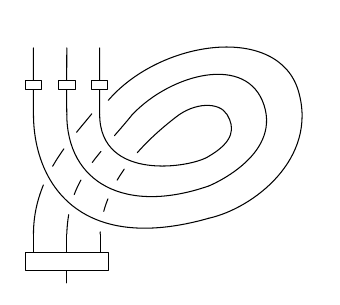}$$

Using the relations of Section \ref{algsec}, we can move the small projectors to the top and absorb the top twists. Using the big projector $P_{3n,k}$ at the bottom to absorb the middle twist leaves us with:
$$\left(\pSHIFT\right)^3\left(\pSHIFT\right)^{-1} \CPic{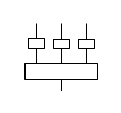}\quad =\quad \sh_k \CPic{enpic}\quad \simeq \CPic{enpic}\hspace{.1in}.$$
The last equality follows because we have reduced the grading, see Section \ref{cyclicgrsec}.
We have shown that the relations $a^2 = 1$ and $b^3 = 1$ hold up to homotopy.
\end{proof}

The corollary below notes that the above theorem defines representations of
the modular group $\slz$ via pullback.

\begin{corollary}
The group $\slz$ acts on the category $\kR_{n,k}$.
\end{corollary}

\begin{proof}
To each element $x\in\slz$ associate the functor $F_{\pi(x)}$ defined in the proof
of the theorem where $\pi : \slz\to \pslz$.
\end{proof}

\subsection{Concluding remarks}

\begin{remark}\label{remark3}

  The argument in Theorem \ref{pslactsthm} applies to the
  subcategories $\kX_n = \Kom^*_{\frac{1}{2}}(n,n,n,n)$ of $\Kom^*_{\frac{1}{2}}(4n)$, see \cite[\S
  6]{CK}. The objects of $\kX_n$ are homotopy equivalent to convolutions of
  chain complexes of the form:
\begin{equation}\label{eqn1}
\CPic{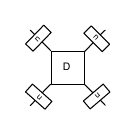}
\end{equation}
where $D$ is a disjoint union of intervals interfacing between four copies
of the projector $P_n$. Reducing the grading so that the functor $\sh_n$
acts by identity gives a new category $\kK_n = \kX_n\sls \sh_n$. After
fixing one projector, act on the other three by $n$-cabled strands as in the
proof above. This is a $\PSL(2,\ZZ)$ representation.

It is possible to argue that the action of $\PSL(2,\ZZ)$ on the category
$\kK_n \ott \FF_2$ is strong by using results from previous papers. This
material has been omitted.

\end{remark}

\begin{remark}
  The objects of the category $\kK_n$ form a basis for the vector space
  assigned to the four punctured sphere $S^2_4$ by $\SU(2)$ Chern-Simons
  theory. The quotient map, $T \to S^2_4$, described in Section
  \ref{pslzsec}, identifies the $\SU(2)$ skein space for $S^2_4$ with the
  $\SO(3)$ skein space for $T^2$ (both with level $r = 2(n+1)$). See
  \cite{BaK, Walker}.
\end{remark}

\begin{remark}
Spin networks of the form (\ref{eqn1}) correspond to $\quantsl$-equivariant
maps from the $n$th irreducible representation $V_n$ to $V_n^{\otimes
  3}$. The dimension of this space is $n+1$. We may view the categories
$\kK_n$ as a sequence of representations of increasing complexity.
\end{remark}

\begin{remark}
  It might be interesting to study the groups $K_0(\kK_n)$ and
  $K_0(\kR_{n,k})$. The groups $K_0(\kP_{n,k})$ and $K_0(\kX_n)$ are
  determined in \cite{CH,CK}, but these groups are not the same because the
  identity $\sh_k \cong \Id$ implies that there must be other relations.
\end{remark}

\begin{remark}
  Tensoring the decomposition of identity in Equation \eqref{decompeq} with
  a chain complex representing the full twist and applying Property
  \ref{abstwprop} shows that the projectors in Theorem \ref{hothm}
  diagonalize the action of the center $\kZ(\B_n)$ on
  $\Kom^*_{\frac{1}{2}}(n)$. In order to reduce an arbitrary $G$-action to a
  $G/\kZ(G)$-action in a more general setting, one must first find an
  extension of $\eC$ in which the action of the center $\kZ(G)$ respects the
  grading. Some results predict such decompositions for other braid group
  representations related to knot homologies \cite[Thm 1.1]{Gorsky}.
\end{remark}

\end{document}

%% file: mypackages.tex
\usepackage{amsmath, amscd, amssymb, amsthm,amsfonts}
\usepackage{stmaryrd,euscript, mathrsfs, latexsym}
\usepackage{color}
\usepackage{hyperref} 

\usepackage{lmodern}
\usepackage[T1]{fontenc} 

\ifx\pdfoutput\undefined
\usepackage{graphicx}
\else
\fi

\ifx\xetex
\usepackage{fontspec}
\usefont{EU1}{lmr}{m}{n}
\else
\fi

\usepackage{tikz}
\usetikzlibrary{matrix,arrows,positioning}
\tikzset{node distance=2cm, auto}

\usepackage[nodayofweek]{datetime}

\parskip=\medskipamount

%% file: mymacros.tex
\DeclareMathOperator{\Hm}{H}

\DeclareMathOperator{\im}{im}

\DeclareMathOperator{\Ho}{Ho}

\DeclareMathOperator{\Cob}{Cob}
\DeclareMathOperator{\PCob}{Pre-Cob}

\DeclareMathOperator{\SL}{SL}

\DeclareMathOperator{\Endo}{End}

\DeclareMathOperator{\Diff}{Diff}

\DeclareMathOperator{\Kom}{Kom}

\DeclareMathOperator{\Mat}{Mat}

\DeclareMathOperator{\Univ}{U}

\DeclareMathOperator{\SU}{SU}
\DeclareMathOperator{\SO}{SO}
\DeclareMathOperator{\Hom}{Hom}

\newcommand{\conj}[1]{\quad\textnormal{ #1 }\quad}
\newcommand{\conjj}[1]{\ \textnormal{ #1 }\ }

\newcommand{\inp}[1]{\ensuremath{\langle #1 \rangle}}

\newcommand{\quantsl}[0]{\Univ_q\mathfrak{sl}(2)}

\makeatletter
\def\imod#1{\allowbreak\mkern2.5mu({\operator@font mod}\,#1)}
\makeatother

\renewcommand{\a}{\alpha}

\newcommand{\vp}{\varphi}

\newcommand{\ott}{\otimes}

\newcommand{\s}{\sigma}

\newcommand{\ga}{\gamma}
\renewcommand{\d}{\delta}

\newcommand{\CPic}[1]{
\begin{minipage}{.45in}
\includegraphics[scale=.75]{#1}
\end{minipage}
}

\newcommand{\bt}{\mathbf{t}}

\newcommand{\rZ}{\mathrm{Z}}

\usepackage{bbm}

\newcommand{\FF}{\mathbb{F}}

\newcommand{\ZZ}{\mathbb{Z}}

\newcommand{\eA}{\EuScript{A}}

\newcommand{\eC}{\EuScript{C}}
\newcommand{\eD}{\EuScript{D}}
\newcommand{\eE}{\EuScript{E}}

\newcommand{\eK}{\EuScript{K}}

\newcommand{\eR}{\EuScript{R}}
\newcommand{\eS}{\EuScript{S}}

\newcommand{\eX}{\EuScript{X}}

\newcommand{\eZ}{\EuScript{Z}}

\theoremstyle{plain}

\newtheorem{theorem}[subsection]{Theorem}

\newtheorem{proposition}[subsection]{Proposition}

\newtheorem{corollary}[subsection]{Corollary}
\theoremstyle{remark}

\newtheorem{remark}[subsection]{Remark}

\theoremstyle{definition}

\newtheorem{example}[subsection]{Example}

\newtheorem{defn}[subsection]{Definition}
\newtheorem{definition}[subsection]{Definition}

\numberwithin{equation}{section}